\documentclass[12pt]{article}
\usepackage[utf8]{inputenc}
\usepackage[T2A]{fontenc}
\usepackage[english]{babel}

\usepackage{parskip}

\usepackage{amsmath}
\usepackage{amssymb}
\usepackage{amsthm}
\usepackage{amsfonts}
\usepackage{asymptote}
\usepackage{hyperref}
\usepackage{xcolor}
\usepackage{enumitem}

\newtheorem{lemma}{Lemma}[]

\newtheorem{proposition}{Proposition}[]

\newtheorem{theorem}{Theorem}[]
\newtheorem{corollary}{Corollary}[]
\theoremstyle{definition}
\newtheorem{definition}{Definition}[]
\newtheorem{example}{Example}[]
\theoremstyle{plain}

\newcommand{\diam}{\operatorname{diam}}

\newcommand{\dis}{\operatorname{dis}}
\newcommand{\dist}{\operatorname{d}}

\newcommand{\Vol}{\operatorname{Vol}}

\renewcommand{\:}{\colon}

\renewcommand{\ss}{\subset}

\newcommand{\N}{\mathbb{N}}
\newcommand{\R}{\mathbb{R}}
\newcommand{\Z}{\mathbb{Z}}

\title{New geodesic lines in the Gromov-Hausdorff class lying in the cloud of the real line.}
\author{I.\,N.\,Mikhailov}
\date{}

\begin{document}
\maketitle
\begin{abstract}
In the paper we prove that, for arbitrary unbounded subset $A\subset R$ and an arbitrary bounded metric space~$X$, a curve $A\times_{\ell^1} (tX)$, $t\in[0,\,\infty)$ is a geodesic line in the Gromov--Hausdorff class. We also show that, for abitrary $\lambda > 1$, $n\in\N$, the following inequality holds: $\dist_{GH}\bigl(\Z^n,\,\lambda\Z^n\bigr)\ge\frac{1}{2}$. We conclude that a curve $t\Z^n$, $t\in(0,\,\infty)$ is not continuous with respect to the Gromov--Hausdorff distance, and, therefore, is not a gedesic line. Moreover, it follows that multiplication of all metric spaces lying on the finite Gromov--Hausdorff distance from $\R^n$ on some~$\lambda > 0$ is also discontinous with respect to the Gromov--Hausdorff distance. 
\end{abstract}

\section{Introduction}
\markright{\thesection.~Introduction}


The Gromov-Hausdorff distance is an important construction of metric geometry, that allows to define a pseudometric on the class of all metric spaces. This distance was first introduced by D.~Edwards in $1975$ in \cite{Edwards} and later became famous thanks to work \cite{GromovGroups}. Explicit historical details can be found in \cite{TuzhilinHistory}.

Traditionally, the Gromov-Hausdorff distance is applied in the studies of compact metric spaces. The space of all compact metric spaces endowed with the Gromov-Hausdorff distance, called \emph{the Gromov--Hausdorff space}, is well-studied. In particular, it is a complete, separable, geodesic metric space.

In the well-known monograph \cite{GromovEng}, Mikhail Gromov described some properties of the space  $\mathcal{GH}$ of all metric spaces, not necessarily compact, considered up to isometry, endowed with the Gromov–Hausdorff distance. In particular,
Mikhail Gromov introduced classes of metric spaces at a finite distance from some fixed metric space (in \cite{BogatyTuzhilin}, such classes were called \emph{clouds}). He announced that such classes are complete and contractible. As a simple example, he gave the Gromov–Hausdorff space. For this metric class, we can consider a natural mapping sending the metric space $(X, d_X)$ to $(X, \lambda d_X)$. If we now let $\lambda$ tend to $0$, we obtain the desired contraction. Mikhail Gromov
indicated that the class of metric spaces at a finite distance from $\R^n$ also has similar properties.

Nevertheless, it turned out later that everything is not so simple. Firstly, when working with the class $\mathcal{GH}$, set-theoretic difficulties arise. It is easy to show, that within the framework of von Neumann–Bernays–G\"odel (NBG) set theory, the space
$\mathcal{GH}$, as well as any cloud, is not a set, but a proper class, i.e., it cannot belong to any other class. To prove that a cloud is contractible, it is necessary to define a topology on it. However, it is impossible
to introduce a topology on a proper class, since any topological space is an element of its topology, which is impossible for a proper class by definition. In \cite{BogatyTuzhilin}, the authors proposed a way to define an analogue of a topology on classes filtered by sets, as well as continuous mappings between them. Secondly, it turned out that not all clouds are invariant under the multiplication of all their metric spaces by an arbitrary number $\lambda > 0$. In \cite{BogatyTuzhilin}, an example of a geometric progression $\{p^n\:n\in\Z\}$ for a prime number $ p > 2$ with a metric induced from $\R$, which bounces off itself by an infinite Gromov–Hausdorff distance under multiplication by~$2$, is given. Finally, while the completeness of an arbitrary cloud was carefully proved in \cite{BogatyTuzhilin}, contractibility has not yet been rigorously justified even for clouds of such natural metric spaces as $\R^n$. The difficulty is to check the continuity of the natural mapping of the multiplication of all spaces of a given cloud by a positive
number $\lambda$.

Another important problem related to the geometry of the Gromov–Hausdorff distance in the class $\mathcal{GH}$ is the problem of constructing geodesics. In \cite{Vihrov}, a class of metric spaces \emph{in general position}, everywhere dense in $\mathcal{GH}$, is presented, any
two of which, at a finite distance from each other, can be connected by a linear geodesic. However, it is still unknown whether any pair of metric spaces, located at a finite distance from each other, can be connected by some geodesic

In this paper, we construct new geodesics in the cloud of the real line. First, we show that if $A \ss \R$ is an unbounded subset, $B \ss \R$ is an arbitrary subset, $A'$ is a metric space at finite Gromov–Hausdorff distance from~$A$, and $X$ and $Y$ are arbitrary bounded metric spaces, then the inequality $d_{GH}\bigl(A' \times_{\ell^1} X, \,B \times_{\ell^1} Y\bigr) \ge \frac{\diam X-\diam Y}{2}$ holds. Using this estimate, we prove that, for an arbitrary bounded metric space $X$ and an unbounded subset $A\ss\R$, the curve $A \times_{\ell^1} (tX),\; t \in [0, \infty)$ is a geodesic in the Gromov–Hausdorff class. After this, we give an example of a space from the cloud $\R^n$ for which multiplication by $t > 0$ will not yield a geodesic in the Gromov–Hausdorff class. Namely, we prove that for arbitrary $n\in\N$, $\lambda > 1$, the inequality $d_{GH}\bigl(\Z^n, \lambda Z^n\bigr) \ge \frac{1}{2}$ holds, which implies that $\Z^n$ is the desired counterexample. Moreover, this inequality implies that the mapping of multiplication of all spaces in the cloud $\R^n$ by an arbitrary $\lambda > 0$ is not continuous with respect to the Gromov--Hausdorff distance. We thus show that the proof of contractibility of the Gromov–Hausdorff space cannot be generalized straightforwardly to the case of the cloud $\R^n$. For simplicity, in this place we do not use the technique of working with classes, filtered by sets, proposed in \cite{BogatyTuzhilin}. By the continuity of the mapping of the multiplication of all spaces of the cloud $[\R^n]$ by $\lambda > 0$ we mean the following natural property: if a sequence of metric spaces $(X_n)_n$ such that $X_n \in [\R^n]$ converges in the Gromov–Hausdorff sense to a metric space~$X$, then $(\lambda X_n)_n$ converges in the Gromov–Hausdorff sense to~$\lambda X$.

\section{Preliminaries}
\markright{\thesection.~Preliminaries}

In this section we provide definitions of the main constructions used, introduce notations, and formulate auxiliary results that we will need in proving the main results.

\subsection{Gromov--Hausdorff distance}

\emph{A metric space} is an arbitrary pair $(X,\,\dist_X)$, where $X$ is an arbitrary set, $\dist_X\: X\times X\to [0,\,\infty)$ is some metric on it, i.\,e. a non-negative, symmetric function that satisfies the triangle inequality.

The distance between any two points $x$ and $y$ of some metric space $(X, d_X)$ will often be denoted by $|xy|$ for convenience. By $U^X_r(a) =\{x\in X\colon |ax|<r\}$, $B^X_r(a) = \{x\in X\colon |ax|\le r\}$ we denote the open and closed balls centered at point~$a$ of radius~$r$ in the metric space~$X$. In those cases when it is clear in which metric space~$X$ the balls are considered, we will omit the superscript. For an arbitrary subset
$A\subset X$ of a metric space~$X$, let $U_r(A) = \cup_{a\in A} U_r(a)$ be an open $r$-neighborhood of~$A$. For non-empty subsets $A \ss X$, $B \ss X$ we put $\dist(A,\,B)=\inf\bigl\{|ab|:\,a\in A,\,b\in B\bigl\}$.

\begin{definition}
Let $A$ and $B$ be non-empty subsets of a metric space.
\emph{The Hausdorff distance} between $A$ and~$B$ is the quantity $$\dist_H(A,\,B) = \inf\bigl\{r > 0\colon A\subset U_r(B),\,B\subset U_r(A)\bigr\}.$$
\end{definition}
\begin{definition} Let $X$ and $Y$ be metric spaces. The triple $(X', Y', Z)$, consisting of a metric space $Z$ and its two subsets $X'$ and $Y'$, isometric to $X$ and~$Y$ respectively, is called \emph{a realization of the pair} $(X, Y)$.
\end{definition}

\begin{definition} \emph{The Gromov-Hausdorff distance} $d_{GH} (X, Y)$ between $X$ and~$Y$ is the exact lower bound of the numbers $r\ge 0$ for which there exists a realization $(X', Y', Z)$ of the pair $(X, Y)$ such that $\dist_H(X',\,Y') \le r$. 
\end{definition}


Now let $X,\,Y$ be non-empty sets.  

\begin{definition} Each $\sigma\subset X\times Y$ is called a \textit{relation} between $X$ and~$Y$.
\end{definition}

By $\mathcal{P}_0(X,\,Y)$ we denote the set of all non-empty relations between $X$ and~$Y$.

We put $$\pi_X\colon X\times Y\rightarrow X,\;\pi_X(x,\,y) = x,$$ $$\pi_Y\colon X\times Y\rightarrow Y,\;\pi_Y(x,\,y) = y.$$ 

\begin{definition} A relation $R\subset X\times Y$ is called a \textit{correspondence}, if restrictions $\pi_X|_R$ and $\pi_Y|_R$ are surjective.
\end{definition}

Let $\mathcal{R}(X,\,Y)$ be the set of all correspondences between $X$ and~$Y$.

\begin{definition} Let $X,\,Y$ be metric spaces, $\sigma \in \mathcal{P}_0(X,\,Y)$. The \textit{distortion} of $\sigma$ is the quantity $$\dis \sigma = \sup\Bigl\{\bigl||xx'|-|yy'|\bigr|\colon(x,\,y),\,(x',\,y')\in\sigma\Bigr\}.$$
\end{definition}

\begin{proposition}[\cite{BBI}]  \label{proposition: distGHformula}
For arbitrary metric spaces $X$ and~$Y$, the following equality holds $$2\dist_{GH}(X,\,Y) = \inf\bigl\{\dis\,R\colon R\in\mathcal{R}(X,\,Y)\bigr\}.$$
\end{proposition}

\subsection{Clouds}

By $\mathcal{VGH}$ denote the class of all non-empty metric spaces, equipped with the Gromov--Hausdorff distance.  

\begin{theorem}[\cite{BBI}]
The Gromov--Hausdorff distance is a generalized pseudometric on $\mathcal{VGH}$, which vanishes on each pair of isometric spaces. Namely, the Gromov--Hausdorff distance is symmetric, satisfies the triangle inequality, but, generally speaking, can be infinite.
\end{theorem}

The class $\mathcal{GH}_0$ is obtained from $\mathcal{VGH}$ by factorization by zero distances, that is, by the equivalence relation: $X\sim_0 Y$ iff $\dist_{GH}(X,\,Y) = 0$.

\begin{definition}
Consider the equivalence relation $\sim_1$ on $\mathcal{GH}_0$: $X\sim_1 Y$ iff $\dist_{GH}(X,\,Y) < \infty$. The corresponding equivalence classes are called \emph{clouds}.
\end{definition}

For an arbitrary metric space $X$, we denote the cloud it defines by $[X]$. By $\Delta_1$ we denote the metric space consisting of a single point. Thus, $[\Delta_1]$ is the cloud consisting of the classes of all bounded spaces at zero Gromov--Hausdorff distance from each other.

\subsection{Cartesian product of metric spaces}

\begin{definition}
Let $X,\,Y$ be two non-empty sets. By $X\times_\rho Y$ we denote the Cartesian product of $X$ and $Y$, equipped with the metric $\rho$.
\end{definition}

Now let $(X,\,\dist_X)$, $(Y,\,\dist_Y)$ be arbitrary metric spaces.  

\begin{definition}
By $X\times_{\ell^1} Y$ we denote the Cartesian product $X\times Y$, equipped with the following $\ell^1$ (or Manhattan) metric $$\dist\bigl(p,\,p'\bigr) = \dist_X(x,\,x') + \dist_Y(y,\,y'),$$ where $p = (x,\,y)$, $p' = (x',\,y')$ are arbitrary points of $X\times Y$.
\end{definition}

\begin{definition} Suppose $A_1,\,A_2$, $B_1,\,B_2$ are non-empty sets. \emph{The Cartesian product} of correspondences $R_1\in\mathcal{R}(A_1,\,A_2)$ and $R_2\in\mathcal{R}(B_1,\,B_2)$ is a correspondence $R_1\times R_2$ between $A_1\times B_1$ and $A_2\times B_2$ defined as follows $$R_1\times R_2 = \Bigl\{\bigl((a_1,\,b_1),\,(a_2,\,b_2)\bigr)\: (a_1,\,a_2)\in R_1,\;(b_1,\,b_2)\in R_2\Bigr\}.$$
\end{definition}

\begin{lemma}[\cite{BBI}, \cite{TuzhilinLectures}]\label{lemma: cartesianproductofcorrespondences}
Let $A_1,\,A_2$, $B_1,\,B_2$ be arbitrary metric spaces, $R_1\in\mathcal{R}(A_1,\,A_2)$, $R_2\in\mathcal{R}(B_1,\,B_2)$. Then \begin{enumerate}
     \item $R_1\times R_2\in\mathcal{R}\bigl(A_1\times_{\ell^1} B_1,\,A_2\times_{\ell^1}B_2\bigr)$;
     \item $\dis R_1\times R_2\le \dis R_1 + \dis R_2$.
 \end{enumerate}
\end{lemma}

We will need the following simple statement

\begin{lemma} \label{correctprod}
Take $A_1,\,A_2\in [X]$ and $B_1,\,B_2\in [Y]$, for two arbitrary metric spaces $X,\,Y$. Then $$\dist_{GH}\bigl(A_1\times_{\ell^1} B_1,\,A_2\times_{\ell^1} B_2\bigr) \le \dist_{GH}(A_1,\,A_2)+\dist_{GH}(B_1,\,B_2).$$
\end{lemma}

\begin{proof}
By the condition, $\dist_{GH}(A_1,\,A_2)<\infty$, $\dist_{GH}(B_1,\,B_2)<\infty$. By Proposition~\ref{proposition: distGHformula}, for any $\varepsilon > 0$ there exist correspondences $R_1\in\mathcal{R}(A_1,\,A_2)$ and $R_2\in\mathcal{R}(B_1,\,B_2)$ such that $\dis R_1 < 2\dist_{GH}(A_1,\,A_2)+\varepsilon$, $\dis R_2 < 2\dist_{GH}(B_1,\,B_2)+\varepsilon$. Then consider the correspondence $R_1\times R_2$ between $A_1\times B_1$ and $A_2\times B_2$. By Lemma \ref{lemma: cartesianproductofcorrespondences}, the following inequality holds $$\dis\bigl(R_1\times R_2\bigr)\le \dis R_1 + \dis R_2 \le 2\bigl(\dist_{GH}(A_1,\,A_2)+\dist_{GH}(B_1,\,B_2)\bigr) + 2\varepsilon.$$ By Proposition \ref{proposition: distGHformula}, since $\varepsilon > 0$ is arbitrary, we obtain $$\dist_{GH}\bigl(A_1\times_{\ell^1} B_1,\,A_2\times B_2\bigr)\le \dist_{GH}(A_1,\,A_2)+\dist_{GH}(B_1,\,B_2),$$ which completes the proof.
\end{proof}

\subsection{A few more auxiliary results}

In this section we present two more statements that we will need in the proofs.

\begin{theorem}[\cite{BBI}, \cite{TuzhilinLectures}]\label{theorem: geodesicinboundedspaces}
The curve $tX$, $t\in[0,\,\infty)$ is a geodesic in the Gromov--Hausdorff class for an arbitrary bounded metric space $X$. In particular, for arbitrary $t_1,\, t_2\ge 0$, the following equality holds 
$$d_{GH}\bigl(t_1X,\,t_2X\bigr) = \frac{|t_1-t_2|}{2}\diam X.$$
\end{theorem}

\begin{theorem}[\cite{HKang}]\label{theorem: integerpointsinunitball}
Consider $\R^n$ with the Euclidean metric. For the number $N(t)$ of points with integer coordinates in the unit ball $B_t(0)\subset\R^n$, the following asymptotic formula holds: $$N(t) = \Vol B_1(0)\cdot t^n \bigl(1+o(1)\bigr),\;t\to\infty,$$
where $\Vol B_1(0)$ is the volume of the unit ball in $\R^n$.
\end{theorem}

\section{Main results}

\begin{theorem} \label{theorem: mainproductinequality}
Suppose $A'\in[A]$, $A\subset\R$ is an unbounded subset, $B\subset\R$, $X,\,Y\in[\Delta_1]$. Then $$\dist_{GH}(A'\times_{\ell^1}X,\,B\times_{\ell^1}Y)\ge\frac{\diam X - \diam Y}{2}.$$
\end{theorem}

\begin{proof}
Let $P = A'\times_{\ell^1}X$, $Q = B\times_{\ell^1}Y$.

If $\dist_{GH}(P,\,Q)=\infty$, then the desired inequality is obvious. Assume that~$\dist_{GH}(P,\,Q) < \infty$. Choose an arbitrary correspondence $R$ between $P$ and $Q$ with finite distortion $c:=\dis R$.

Let also $x_0,\,x_1\in X$ be such that $|x_0x_1| = t$.

Since $A'\in[A]$, there exists a correspondence $S$ between $A$ and $A'$ with distortion $\dis S = w<\infty$.

Since $A\subset\R$ is an unbounded subset, there exist points $p_1<p_2<\,\ldots\,<p_{2n+1}$ in $A$ such that $d_i:=|p_ip_{i+1}|> 100(t+c+w+\diam Y)$ for all $i = 1,\,\ldots,\,2n+1$. Choose arbitrary $a_i\in S(p_i)$.

We put $A_{ij} = (a_i,\,x_j)\in A'\times_{\ell^1}X,\,i = 1,\,\ldots,\,2n+1$, $j = 0,\,1$.

Note that for all $1\le i < k \le 2n+1$, $j\in\Z/2\Z$ the following equalities hold
\begin{align*} |A_{ij}A_{k\,j+1}| = |a_ia_k|+t,\;|A_{ij}A_{kj}| = |a_ia_k|.
\end{align*}

\begin{figure}[ht]
\centering 
\begin{asy}[width = 0.5\linewidth]
size(5cm, 0);
import geometry;
real degree = 2;
real mr = 0.45cm;

point pA = (-6, 0);
point pB = (-6, 2);
point pC = (-2, 2);
point pD = (-2, 0);
point pE = (2, 2);
point pF = (2, 0);

draw(line(pA, pF), dashed);
draw(line(pB, pC), dashed);

dot("$A_{10}$", pA, S);
dot("$A_{11}$", pB, N);
dot("$A_{21}$", pC, N);
dot("$A_{20}$", pD, S);
dot("$A_{30}$", pF, S);
dot("$A_{31}$", pE, N);

draw("$t$", pA--pB, W);
draw("$|a_1a_2|$", pB--pC, N, fontsize(8pt));
draw("$t$", pC--pD, E);
draw("$|a_1a_2|$", pD--pA, S, fontsize(8pt));
draw("$|a_2a_3|$", pC--pE, N, fontsize(8pt));
draw("$|a_2a_3|$", pD--pF, S, fontsize(8pt));
draw("$t$", pE--pF, W);

\end{asy}
\end{figure}

We choose $B_{ij} = (b_{ij},\,y_{ij})\in R(A_{ij})$ arbitrarily.

Note that the inequalities hold
\begin{align*}
|B_{ik}B_{jl}|\ge|b_{ik}-b_{jl}| = |B_{ik}B_{jl}| - \dist_Y(y_{ik},\,y_{jl}) \ge |B_{ik}B_{jl}| - \diam Y.
\end{align*}

\begin{lemma}
For arbitrary $1\le i < j < k \le 2n+1$ and $\alpha,\,\beta,\,\gamma$ from $\Z/2\Z$, the point $b_{j\beta}$ lies strictly between $b_{i\alpha}$ and $b_{k\gamma}$.
\end{lemma}

\begin{proof}
By the definition of distortion, the following inequalities hold:

\begin{multline*}
|b_{i\alpha}b_{k\gamma}|\ge |B_{i\alpha}B_{k\gamma}|-\diam Y\ge |A_{i\alpha}A_{k\gamma}|-(\diam Y+c)\ge\\\ge |a_ia_k|-(\diam Y+c)\ge |p_ip_k|-(\diam Y+c+w)=\\= (d_i + \ldots + d_{k-1}) - (\diam Y+c+w).
\end{multline*}
Assume that the statement being proven is false. Then
\begin{multline*}
|b_{i\alpha}b_{k\gamma}| = \bigl||b_{i\alpha}b_{j\beta}| - |b_{j\beta}b_{k\gamma}|\bigr|\le\\\le \max\bigl\{|b_{i\alpha}b_{j\beta}|,\,|b_{j\beta}b_{k\gamma}|\bigr\} \le \max\bigl\{|B_{i\alpha}B_{j\beta}|,\,|B_{j\beta}B_{k\gamma}|\bigr\}\le\\ \le \max\bigl\{|A_{i\alpha}A_{j\beta}|+c,\,|A_{j\beta}A_{k\gamma}|+c\bigr\}\le \max\bigl\{|a_ia_j|+t+c,\,|a_ja_k|+t+c\bigr\}\le \\\le\max\bigl\{|p_ip_j|+t+c+w,\,|p_jp_k|+t+c+w\bigr\} =\\= \max\bigl\{(d_i+\ldots+d_{j-1})+w+t+c,\,(d_j+\ldots+d_{k-1})+w+t+c\bigr\} <\\< (d_i+\ldots + d_{k-1})-w-c-\diam Y,
\end{multline*}
where the last inequality is satisfied, since $d_l > 2(w+c)+t+\diam Y$ for each $l$. That is a contradiction.
\end{proof}

Thus, without loss of generality, we can assume that pairs of points $\{b_{ij},\,b_{i\,j+1}\}$ are located on a line in ascending order of indices $i$.

\begin{figure}[ht]
\centering 
\begin{asy}[width = 0.5\linewidth]
size(5cm, 0);
import geometry;
real degree = 2;
real mr = 0.45cm;

point pA = (0, 0);
point pB = (1, 0);
point pC = (6, 0);
point pD = (7, 0);
point pE = (12, 0);
point pF = (13, 0);

dot("$b_{10}$", pA, S, fontsize(8pt));
dot("$b_{11}$", pB, NE, fontsize(8pt));
dot("$b_{21}$", pC, S, fontsize(8pt));
dot("$b_{20}$", pD, NE, fontsize(8pt));
dot("$b_{31}$", pE, S, fontsize(8pt));
dot("$b_{30}$", pF, NE, fontsize(8pt));

draw(pA--pB);
draw(pC--pD);
draw(pE--pF);

//label("$\ge d_1-c'-w$", pB--pC, S, fontsize(3pt));
//label("$\ge d_2-c'-w$", pD--pE, S, fontsize(3pt));
//label("$d/3$", pA--pB, N, fontsize(8pt));
//label("$d/3$", pC--pD, N, fontsize(8pt));
//label("$d/3$", pE--pF, N, fontsize(8pt));

draw(line(pA, pD), dashed);
\end{asy}
\label{Example1B}
\end{figure}

We choose indices $i_1,\,i_2,\,\ldots,\,i_{2n+1}\in\Z/2\Z$ such that \\
1) $b_{1i_1}$ is the leftmost point of all $b_{ij}$, $i = 1,\,\ldots,\,2n+1$, $j=0,\,1$;\\ 2) $|A_{ji_j}A_{j+1\,i_{j+1}}| = t+|a_ja_{j+1}|$ for each $j = 1,\,\ldots,\,2n$.

Then the following inequalities hold
\begin{multline*}
 c+w +(d_1+\ldots + d_{2n}) = c+w+|p_1p_{2n+1}|\ge c + |a_{1i_1}a_{2n+1\,i_{2n+1}}| =\\= c + |A_{1i_1}A_{2n+1\,i_{2n+1}}|\ge |B_{1i_1}B_{2n+1\,i_{2n+1}}|\ge  |b_{1i_1}b_{2n+1\,i_{2n+1}}|=\sum_{k = 1}^{2n} |b_{ki_k}b_{k+1\,i_{k+1}}|\ge\\\ge \sum_{k=1}^{2n}\bigl(|B_{ki_k}B_{k+1\,i_{k+1}}|-\diam Y\bigr)\ge\sum_{k = 1}^{2n} \bigl(|A_{ki_k}A_{k+1\,i_{k+1}}| - c - \diam Y\bigr) =\\= \sum_{k=1}^{2n} |a_ka_{k+1}| +  (t-c-\diam Y)\cdot 2n \ge |a_{1i_1}a_{2n+1\,i_{2n+1}}| + (t-c-\diam Y)\cdot 2n\ge\\\ge |p_{1i_1}p_{2n+1\,i_{2n+1}}|-w + (t-c-\diam Y)\cdot 2n= (d_1+\ldots +d_{2n})-w + (t-c-\diam Y)\cdot 2n. 
\end{multline*}
We obtain that $$c+\frac{2w}{2n+1}+\frac{2n}{2n+1}\diam Y\ge \frac{2n}{2n+1}t.$$

Since $t$ can be chosen arbitrarily close to $\diam X$, and $n$ can be chosen arbitrarily large, we obtain that $c\ge \diam X - \diam Y$.

Now, due to the arbitrariness of the chosen correspondence $R$, the desired estimate follows from the Proposition~\ref{proposition: distGHformula}.
\end{proof}

\begin{corollary}\label{corollary: productinequalities}
$1)$ Suppose $A,\,B\subset\R$, $X,\,Y\in[\Delta_1]$, $A$ and $B$ are unbounded. Then $$\dist_{GH}(A\times_{\ell^1}X,\,B\times_{\ell^1}Y)\ge\Bigl|\frac{\diam X - \diam Y}{2}\Bigr|.$$
$2)$ If $A\in [\R]$, $X\in[\Delta_1]$, then $$\dist_{GH}\bigl(A\times_{\ell^1}X,\,\R\bigr)\ge \frac{\diam X}{2}.$$
\end{corollary}

\begin{proof}
1) Since $A,\,B\subset \R$ are unbounded subsets, according to Theorem \ref{theorem: mainproductinequality} the following inequalities hold
\begin{align*}
&\dist_{GH}(A\times_{\ell^1}X,\,B\times_{\ell^1}Y)\ge\frac{\diam X - \diam Y}{2},\\
&\dist_{GH}(B\times_{\ell^1}Y,\,A\times_{\ell^1}X)\ge\frac{\diam Y - \diam X}{2},
\end{align*}
from which the desired inequality follows.

2) Applying the inequality of Theorem \ref{theorem: mainproductinequality} for the spaces $A\times_{\ell^1}X$ and $\R\times_{\ell^1}\Delta_1$, we obtain the desired inequality.
\end{proof}

\begin{corollary}
Let $X$ be a bounded metric space and $A\subset \R$ be an unbounded subset. Then for any positive $t_1,\,t_2$ the equality holds $$\dist_{GH}\bigl(A\times_{\ell^1}(t_1X),\,A\times_{\ell^1}(t_2X)\bigr) = \frac{|t_1-t_2|}{2}\diam X.$$
\end{corollary}

\begin{proof}
We put $A_t = A\times_{\ell_1}(tX)$. Then it follows from Theorem~\ref{theorem: mainproductinequality} that $\dist_{GH}(A_{t_1},\,A_{t_2})\ge \frac{|t_1-t_2|\diam X}{2}$. On the other hand, by Theorem~\ref{theorem: geodesicinboundedspaces} the equality $\dist_{GH}\bigl(t_1X,\,t_2X\bigr) = |t_1-t_2|\diam X$ holds. Then it follows from Lemma~\ref{correctprod} that $\dist_{GH}(A_{t_1},\,A_{t_2})\le\frac{|t_1-t_2|\diam X}{2}$. Therefore, $\dist_{GH}(A_{t_1},\,A_{t_2})=\frac{|t_1-t_2|}{2}\diam X$, which is what was required to be proved.
\end{proof}

\begin{corollary}
For an arbitrary bounded metric space $X$, the curve $\{\R\}\cup\bigl\{t(\R\times_{\ell^1}X)\: t\in(0,\,+\infty)\bigr\}$ is a geodesic in the Gromov--Hausdorff class.  
\end{corollary}

\begin{example}
Let us show that for arbitrary $P=A\times_{\ell^1}X$ and $Q=B\times_{\ell^1}Y$ such that $A,\,B\in[\R]$, $X,\,Y\in[\Delta_1]$, the estimate $\dist_{GH}(P,\,Q)\ge |\diam(X)-\diam(Y)|/2$ is, in general, false.

Consider $P = (\R+c)\times_{\ell_1} [0,1]$ and $Q = \R\times_{\ell^1} \bigl([0,1]+c\bigr)$. Here $\R+c$, $[0,1]+c$ are the spaces $\R$, $[0,1]$ with metrics induced from $\R$ and increased by the constant $c>0$. The spaces $P$ and $Q$ are isometric via the identity mapping, although the difference in the diameters of their bounded factors is equal to $c$.
\end{example}

\begin{theorem}
For arbitrary $\lambda > 1$, $n\in\N$ the inequality $$\dist_{GH}\bigl(\Z^n,\,\lambda\Z^n\bigr)\ge\frac{1}{2}.$$
\end{theorem}

\begin{proof}
Let $R\in\mathcal{R}\bigl(\Z^n,\,\lambda\Z^n\bigr)$ be a correspondence with distortion $c:=\dis R < \infty$. Since the shift of $\Z^n$ by an integer vector is an isometry, we can assume without loss of generality that $(0,\,0)\in R$.

Assume that $R$ is bijective. Then consider the ball $B = B^{\Z^n}_{\lambda t}(0)$ for some $t > 0$. Note that $R(B)\subset B' = B^{\lambda\Z^n}_{\lambda t+c}(0)$. By $N(t)$, $N'(t)$ we denote the numbers of points in the balls $B$ and $B'$, respectively. By Theorem~\ref{theorem: integerpointsinunitball} for $t\to\infty$ the following equalities hold
\begin{align}
&N(t) = \Vol B^{\R^n}_1(0) \lambda^nt^n\bigl(1+o(1)\bigr), \label{asimptotic1}\\
&N'(t) = \Vol B^{\R^n}_1(0)\Bigl(t+\frac{c}{\lambda}\Bigr)^n\bigl(1+o(1)\bigr)\label{asimptotic2}.
\end{align}
Since $R$ is bijective, it follows from the inclusion $R(B)\subset B'$ that $N'(t)\ge N(t)$. However, from the formulas $(\ref{asimptotic1})$, $(\ref{asimptotic2})$ it follows that $\lim_{t\to\infty} \frac{N(t)}{N'(t)} = \lambda^n > 1$ --- a contradiction.

So, $R$ is not bijective. Then $c \ge 1$. Since $R$ is arbitrary, by the Proposition~\ref{proposition: distGHformula} we obtain the desired inequality.
\end{proof}

\begin{corollary}
The mapping $[\R^n]\times (0;\,+\infty)\to[\R^n]$, $(A,\,\lambda)\to \lambda A$ is not continuous in $\lambda$.
\end{corollary}

\begin{corollary}
The curve $\lambda \Z^n$, $\lambda\in(0;\,\infty)$ is not continuous in the Gromov--Hausdorff class. In particular, it is not a geodesic.
\end{corollary}

\clearpage


\end{document}